\documentclass[preprint,12pt]{article}
\usepackage{graphicx}
\usepackage{authblk}

\usepackage{graphicx}
\usepackage{float}

\usepackage{multicol}
\usepackage{amsmath}
\usepackage{amssymb}
\usepackage{amsthm}
\newtheorem{theorem}{Theorem}

\newtheorem{proposition}[theorem]{Proposition}

\newtheorem{problem}{Problem}

\begin{document}

\title{Minimal digraph obstructions for small matrices\thanks{Both authors were supported by the NSERC Discovery Grant of the first author, who was additionally supported by the grant ERCCZ LL 1201. Also, part of this work was done while the first author was visiting the Simons Institute for the Theory of Computing.}}

\author{Pavol Hell\thanks{email: pavol@sfu.ca} \\ {\small School of Computing Science \\ Simon Fraser University} \and
C\'esar Hern\'andez-Cruz\thanks{email: chc@ciencias.unam.mx (Corresponding Author)} \\ {\small Instituto de Matem\'aticas\\Universidad Nacional Aut\'onoma de M\'exico} \and
}

\maketitle
\begin{abstract}
Given a $\{ 0, 1, \ast \}$-matrix $M$, a minimal
$M$-obstruction is a digraph $D$ such that $D$
is not $M$-partitionable, but every proper induced
subdigraph of $D$ is.   In this note we present a
list of all the $M$-obstructions for every $2 \times
2$ matrix $M$.

Notice that this note will be part of a larger paper,
but we are archiving it now so we can cite the results.
\end{abstract}

\section{Introduction}

Given an $m$ by $m$ matrix $M$ over $0, 1, *$
(a {\em pattern}), an $M$-{\em partition} of a
digraph $G$ is a partition of the vertices into
parts $V_1, V_2, \dots, V_m$ such that two
distinct vertices in $V_i$ are non-adjacent if
$M_{i,i}=0$, and adjacent in both directions if
$M_{i,i}=1$ ($M_{i,i}=*$ represents no restriction).
Similarly, each vertex in $V_i$ must (respectively
must not) dominate each vertex in $V_j$ if $M_{i,j}
=1$ (respectively $M_{i,j}=0$).

Given a pattern $M$, the {\em $M$-partition
problem} is the decision problem of determining
whether a digraph admits an $M$-partition.
Notice that if we regard (undirected) graphs as
digraphs in which every arc is a digon, the
$M$-partiton problem is also defined for graphs
when $M$ is a symmetric matrix.   Many well
known problems in graph theory can be posed
as $M$-partition problems, e.g., a $\left(
\begin{array}{cc} 0 & \ast \\ \ast & 0 \end{array}
\right)$-partition is just a bipartition, and a $\left(
\begin{array}{cc} 0 & \ast \\ \ast & 1 \end{array}
\right)$ is a split partition.

Observe that having an $M$-partition is an
hereditary property, and thus, $M$-partitionable
digraphs can be characterized through a set of
forbidden induced subdigraphs.   A {\em minimal
obstruction to $M$-partition}, or {\em minimal
$M$-obstruction} for short, is a digraph that does
not admit an $M$-partition but such that every
proper induced subdigraph does.   Clearly, if
a pattern $M$ has only a finite number of minimal
obstructions, then the $M$-partition problem is
polynomial time solvable.   Nonetheless, there
are patterns $M$ with infinitely many minimal
$M$-obstructions and with a polynomial time
solvable $M$-partition problem, e.g. the
bipartition problem.

There are two main
problems associated with the concept of an
$M$-partition.

\begin{problem}[The Characterization Problem]
Which patterns $M$ have the property that the
number of mimimal $M$-obstructions is finite.
\end{problem}

\begin{problem}[The Complexity Problem]
Which patterns $M$ have the property that the
$M$-partition problem can be solved by a
polynomial time algorithm?
\end{problem}

We refer the reader to \cite{survey} for a survey
on the subject.

The main goal of this note is to give a full list of
minimal $M$-obstructions for every $2 \times 2$
pattern $M$.   Although it is already known that
the $M$-partition problem for such patterns is
polynomial time solvable, it is useful to have the
exact list of minimal $M$-obstructions.   Such
list has been already used in \cite{PointDet}
(where it was meant to be originally included),
and recently in \cite{HKerComp}.

We refer the reader to \cite{BJD} and \cite{BM}
for general concepts.   In this work, $D = (V_D,
A_D)$ will be a digraph with the vertex set $V_D$
and the arc set $A_D$, without loops of multiple
arcs in the same direction.   If $A_D = \varnothing$
we say that $D$ is an {\em empty digraph}.   The
{\em dual} of $D$ is the digraph $\overleftarrow{D}$
obtained from $D$ by reversing each of its arcs.
We will denote the underlying graph of $D$ by
$G_D$.   The {\em complement} of $D$ is the
digraph $\overline{D}$ with $V_{\overline{D}} =
V_D$ and such that $(x,y) \in A_{\overline{D}}$
if and only if $(x,y) \notin A_D$.

When $(x,y) \in A_D$ ($(x,y) \notin A(D)$) we will
denote it by $x \to y$ ($x \not \to y$).   We will say
that an arc $(x,y)$ is a {\em digon} if $y \to x$;
otherwise, we will say that $(x,y)$ is an {\em
asymmetric} arc.   If $x \to y$ we say that $x$ is
an {\em in-neighbour of $y$} and $y$ is an {\em
out-neighbour of $x$}.   The {\em in-neigbourhood}
({\em out-neighbourhood}) of a vertex $v$, $N^-(v)$
($N^+(v)$) is the set of all its in-neighbours
(out-neighbours).   The {\em neighbourhood} of
$v$, $N(v)$ is defined as $N(v) = N^- (v) \cup N^+
(v)$.   For vertices $x$ and $y$, we say that
$x$ is {\em adjacent} to $y$ if $y \in N(x)$.

Given a graph $G$, a {\em superorientation} of
$G$ is obtained by replacing each edge $xy$ in
$G$ with $(x,y)$, $(x,y)$ or both of them.   An
{\em orientation} of $G$ is a superorientation
of $G$ without digons.   A {\em biorientation} of
$G$ is a superorientation of $G$ where every
arc is a digon; the (unique up to isomorphism)
biorientation of $G$ is denoted by
$\overleftrightarrow{G}$.   A subset $S$ of $V_D$
is a {\em strong clique} if it induces a biorientation
of a complete graph in $D$.   We will often abuse
language and say that a strong clique on two
vertices is a digon.   A digraph is {\em strict split}
if $V_D$ admits a partition $(V_0, V_1)$ such
that $V_0$ is an independent set and $V_1$
is a strong clique.

The disjoint union of $D_1$ and $D_2$ is denoted
by $D_1 + D_2$.

\section{Main results}

It follows from \cite{FHKM} that for every two by two
matrix $M$ the recognition of $M$-partitionable
digraphs is possible in polynomial time by reducing
the problem to 2-SAT. If $M$ has an asterisk on
the main diagonal, every digraph has an $M$-partition.
To reduce the possibilities when there are no
asterisks on the main diagonal, we present two
simple results.

Let $\overline{M}$ denote the pattern obtained
from $M$ by replacing each entry $0$ by $1$
and vice versa.   The following result is easy to
verify.

\begin{proposition} \label{complement}
A partition of $V_D$ is an
$\overline{M}$-partition of $D$ if and only if
it is an $M$-partition of $\overline{D}$.
\end{proposition}

A similar result can be obtained for
$\overleftarrow{D}$ and the transpose $M^t$
of $M$.

\begin{proposition} \label{transpose}
A partition of $V_D$ is an $M^t$-partition
of $D$ if and only if it is an $M$-partition of
$\overleftarrow{D}$.
\end{proposition}

It follows from Proposition \ref{complement}
that a digraph $D$ is a minimal $M$-obstruction
if and only if $\overline{D}$ is a minimal
$\overline{M}$-obstruction.   Analogously,
it follows from Proposition \ref{transpose}
that a digraph $D$ is a minimal $M$-obstruction
if and only if $\overleftarrow{D}$ is a minimal
$M^t$-obstruction.

There are $36$ different $2 \times 2$ patterns
with the required properties.   Nonetheless, it
follows from Propositions \ref{complement}
and \ref{transpose}, and a simple additional
analysis when $M_{11} = 0$ and $M_{22}
= 1$, that there are exactly $10$ such
patterns with essentially different sets of minimal
obstructions, $$M_1 =  \left( \begin{array}{cc} 0
& \ast \\ \ast & 0 \end{array} \right), \hspace{1cm}
M_2 = \left( \begin{array}{cc} 0 & 0 \\ \ast & 0
\end{array} \right), \hspace{1cm} M_3 = \left(
\begin{array}{cc} 0 & 1 \\ \ast & 0 \end{array}
\right),$$ $$M_4 = \left( \begin{array}{cc} 0 & 0
\\ 1 & 0 \end{array} \right), \hspace{1cm} M_5 =
\left( \begin{array}{cc} 0 & 1 \\ 1 & 0 \end{array}
\right), \hspace{1cm} M_6=\left( \begin{array}{cc}
0 & 0 \\ 0 & 0 \end{array} \right),$$ $$M_7 = \left(
\begin{array}{cc} 0 & \ast \\ \ast & 1 \end{array}
\right), \hspace{1cm} M_8 = \left( \begin{array}{cc}
0 & 0 \\ \ast & 1 \end{array} \right), \hspace{1cm}
M_9 = \left( \begin{array}{cc} 0 & 1 \\ 0 & 1 \end{array}
\right),$$ $$M_{10} = \left( \begin{array}{cc} 0 & 0 \\
0 & 1 \end{array} \right).$$

Thus, it suffices to analyze the minimal
$M_i$-obstructions for $1 \le i \le 10$.
We begin by analyzing the case when
our pattern has only zeros in the main
diagonal.   Notice first that a digraph
admits an $M_1$-partition if and only
if it is bipartite, which happens if and
only if its underlying graph is bipartite.
Thus, the digraph minimal
$M_1$-obstructions are all the possible
superorientations of every (undirected)
odd cycle.   We will show in the following
theorem that this is the only $2 \times 2$
matrix with zero diagonal and infinitely
many minimal obstructions.   The sets
of minimal $M_i$-obstructions for
$2 \le i \le 6$ are depicted in Figure
\ref{0**0Fig}.   In this figure, an edge
between a pair of vertices means that
an arc must be present between them,
and it can be oriented either way or as
a digon.   For each $i \in \{2, \dots, 6 \}$,
let us refer to the digraphs corresponding
to $M_i$ in Figure \ref{0**0Fig} as
$\mathcal{F}_i$.

\begin{theorem} \label{0**0}
Suppose $M$ is a two by two matrix with
zero diagonal different from $M_1$.
There are a finite number of minimal
$M$-obstructions, which are depicted in
Figure \ref{0**0Fig} for every possible $M$.
\begin{figure}[H]
\begin{center}
\includegraphics[width=\textwidth]{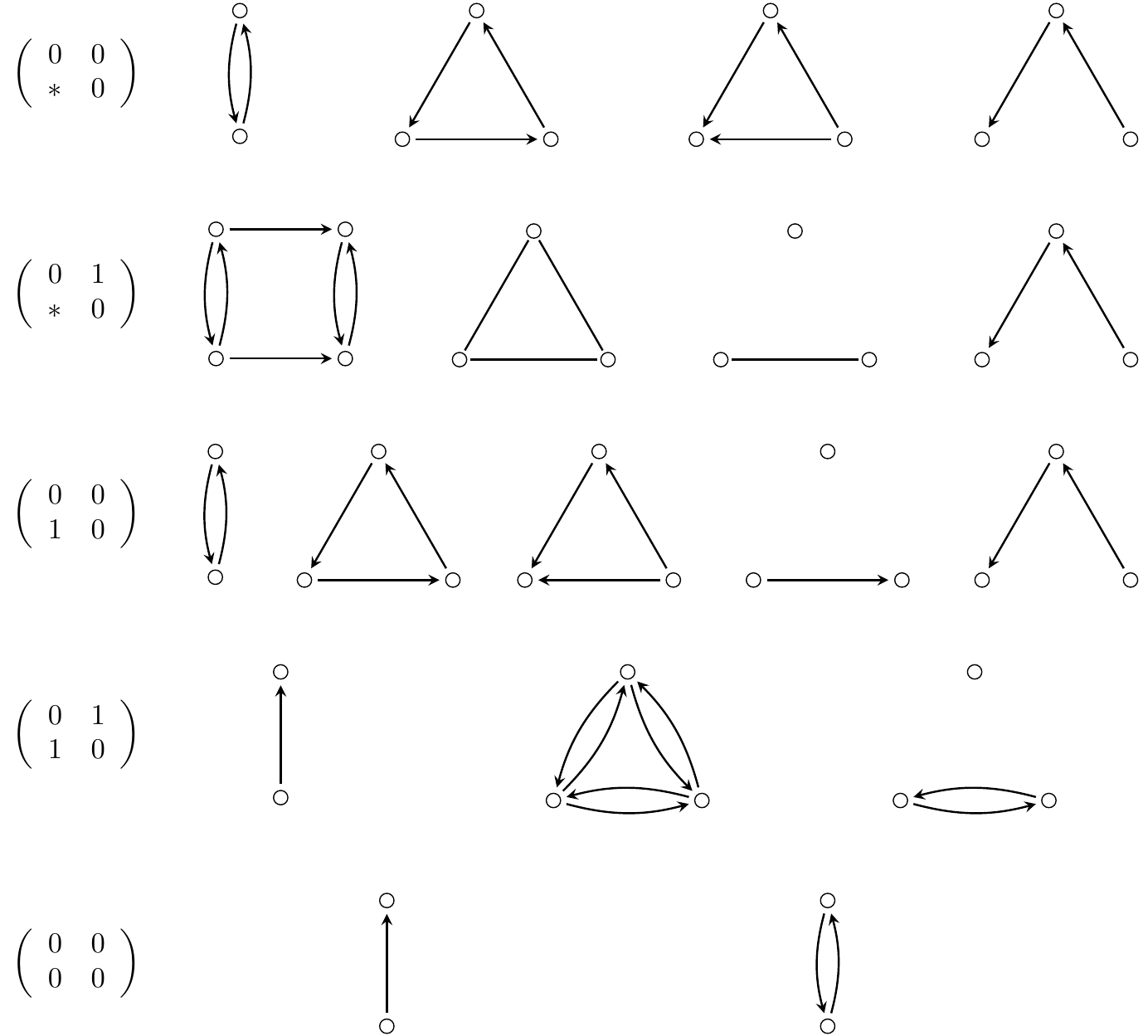}
\end{center}\caption{Minimal obstructions
for matrices in Theorem \ref{0**0}.} \label{0**0Fig}
\end{figure}
\end{theorem}

\begin{proof}
It is easy to verify that members of
$\mathcal{F}_i$ are minimal $M_i$-obstructions
for $2 \le i \le 6$. It is clear as well that every
other digraph on two or three vertices is not
a minimal $M_i$-obstruction.

Suppose that $D$ is a minimal
$M_2$-obstruction on at least $4$ vertices.
Since a digon is a minimal $M_2$-obstruction,
we can assume that $D$ is an oriented graph.
Let $v$ be an arbitrary vertex of $D$, and
let $(V_1, V_2)$ be an $M_2$-partition of
$D - v$.   Observe that the neighbourhood of
$v$ must be an independent set, otherwise,
there would be a tournament on three
vertices properly contained in $D$, which
is already a minimal $M_2$-obstruction.
Since the directed path of length $2$ is also
a minimal $M_2$-obstruction, then either
$N^+ (v) = \varnothing$ or $N^- (v) =
\varnothing$.   It cannot be the case that
the neighbourhood of $v$ is empty, because
$(V_1 \cup \{ v \}, V_2)$ would be an
$M_2$-partition of $D$.   Suppose first
that $N^+ (v) = \varnothing$.   If $N^- (v)
\subseteq V_2$, then $(V_1 \cup \{ v \},
V_2)$ is an $M_2$-partition, a contradiction.
Hence, $N^- (v) \cap V_1 \ne \varnothing$.
If some vertex in $N^- (v) \cap V_1$ has
positive indegree, then there is an induced
directed path of length $2$ in $D$,
contradicting the minimality of $D$.   Thus,
$V_2' = V_2 \cup (N^-(v) \cap V_1)$ is an
independent set.   If $V_1' = (V_1 \cup \{ v \})
\setminus N^- (v)$, then $(V_1',  V_2')$
is an $M_2$-partition of $D$, a contradiction.
Thus, $N^+ (v) \ne \varnothing$.   A very
similar argument shows that we also reach
a contradiction when $N^- (v) = \varnothing$.
Since the contradiction comes from assuming
that there exists a minimal $M_2$-obstruction
of order greater than three, we conclude that
all the minimal $M_2$-obstructions have order
at most three, and hence, they are precisely
the digraphs in $\mathcal{F}_2$.

Clearly, an $M_5$-partitionable digraph is
either an empty digraph or a biorientation
of a complete bipartite graph.   Thus, it suffices
to show that a non-empty graph which contains
neither $K_3$, nor $K_1 + K_2$ as an induced
subgraph, is a complete bipartite graph.
Let $G$ be such a graph.   Since $G$ does
not contain $K_1 + K_2$ an an induced
subgraph, it is clear that $G$ does not
contain any cycle of length greater than
$4$ as an induced subgraph.   Thus, $G
= (X, Y)$ is a bipartite graph.   If $|X| = 1$
or $|Y| = 1$, it is direct to verify that $G$
is complete bipartite.   Since $G$ is
non-empty, we can choose an edge $xy$
of $G$ with $x \in X$ and $y \in Y$.   Let
$x' \in X$ and $y' \in Y$ be arbitrarily chosen.
The edges $x'y$ and $xy'$ must be present
in $G$, otherwise, a $K_1 + K_2$ would be
an induced subgraph of $G$.   But then, the
edge $x' y'$ must also be present in $G$,
else, $\{ x, x', y' \}$ would induce a $K_1 +
K_2$ in $G$.   Since the choice of $x'$ and
$y'$ is arbitrary, we conclude that $G$ is a
complete bipartite graph.

Let $D$ be an $\mathcal{F}_4$-free digraph.
Since the digon is an element of
$\mathcal{F}_4$, $D$ is an oriented graph.
Moreover, the underlying graph of $D$, $G_D$,
contains neither $K_3$ nor $K_1 + K_2$,
because $\mathcal{F}_4$ contains the two
tournaments on three vertices and the digraph
on three vertices consisting of an isolated
vertex and an arc.   Thus, following the
argument used in the previous case, we
conclude that $D$ is either an empty digraph,
or $G_D = (X, Y)$ is a complete bipartite graph.
In the latter case, every arc of $D$ is oriented,
without loss of generality, from $X$ to $Y$,
otherwise, there would be a directed path
of length $2$ as an induced subdigraph of
$D$, but $D$ is $\mathcal{F}_4$-free.   Thus,
$(X, Y)$ is an $M_4$-partition of $D$.

If $D$ is an $\mathcal{F}_3$-free digraph,
then, as in the two previous cases, we derive
that either $D$ is an empty digraph, or the
underlying graph of $D$, $G_D = (X, Y)$, is
a complete bipartite graph.   If all the arcs
of $D$ are digons, then $(X, Y)$ is an
$M_3$-partition, and we are done.   Assume
without loss of generality that $x \to y$ and
$y \not \to x$ for some $x \in X$ and $y \in Y$.
Suppose for a contradiction that there are
$x' \in X$ and $y' \in Y$ such that $y' \to x'$
and $x' \not \to y'$.   Since $D$ does not
contain directed paths of length $2$ as
induced subdigraphs, we have $x \ne x'$
and $y \ne y'$; with the same argument we
conclude that the arcs $(x, y')$ and $(x', y)$
are digons.   But, $\{ x, x', y, y' \}$ induces
the only digraph on four vertices in
$\mathcal{F}_3$, a contradiction. Therefore,
all the arcs from $X$ to $Y$ are present
in $D$, and thus, $(X, Y)$ is an
$M_3$-partition of $D$.

When $M = M_6$, an $M$-partitionable
digraph $D$ is just an empty graph.   Hence,
it is clear that the only minimal
$M_6$-obstructions are an asymmetric
arc and a digon, this is, $\mathcal{F}_6$.
\end{proof}

We conclude this note with the analysis of
the patterns having a $0$ and a $1$ in the
main diagonal.   The pattern $M_7$ has been
already studied in \cite{strict}; an $M_7$-partition
corresponds to a strict split partition, this is, a
partition $(V_0, V_1)$ with $V_0$ an independent
set and $V_1$ a strong clique.   Refer to \cite{strict},
for the complete list of minimal $M_7$-obstructions.
Together with the aforementioned result for $M_7$,
our next result shows that for these kind of patterns,
there are always finitely many minimal obstructions.

For $8 \le i \le 10$, it is clear that
$\overleftrightarrow{K_3}$ and its complement are
$M_i$-partitionable.   If $i \in \{ 8, 9 \}$, then it is not
hard to verify that every other digraph on three
vertices is a minimal $M_i$-obstruction, except for
the digraphs depicted in Figure \ref{01DF}.   If $i =
10$, then clearly the asymmetric arc is a minimal
$M_i$-obstruction, hence, every other obstruction
is a biorientation of some (undirected) graph.
Define $\mathcal{F}_i$ in the following way.
\begin{itemize}
	\item $\mathcal{F}_8$ consists of the
		aforementioned minimal $M_8$-obstructions
		on $3$ vertices, together with all the
		superorientations of $2K_2$.
	
	\item $\mathcal{F}_9$ consists of the
		aforementioned minimal $M_9$-obstructions
		on $3$ vertices.
		
	\item $\mathcal{F}_{10}$ consists of the
		asymmetric arc,
		$\overleftrightarrow{2K_2}$, and the
		biorientations of all graphs on three
		vertices, except for $K_1 + K_2$.
\end{itemize}

It is direct to verify that every digraph in $\mathcal{F}_i$
is indeed a minimal $M_i$-obstruction, for $8 \le i \le 10$.
It comes as no surprise that these are the only minimal
$M_i$-obstructions.

\begin{figure}[h]
\begin{center}
\includegraphics{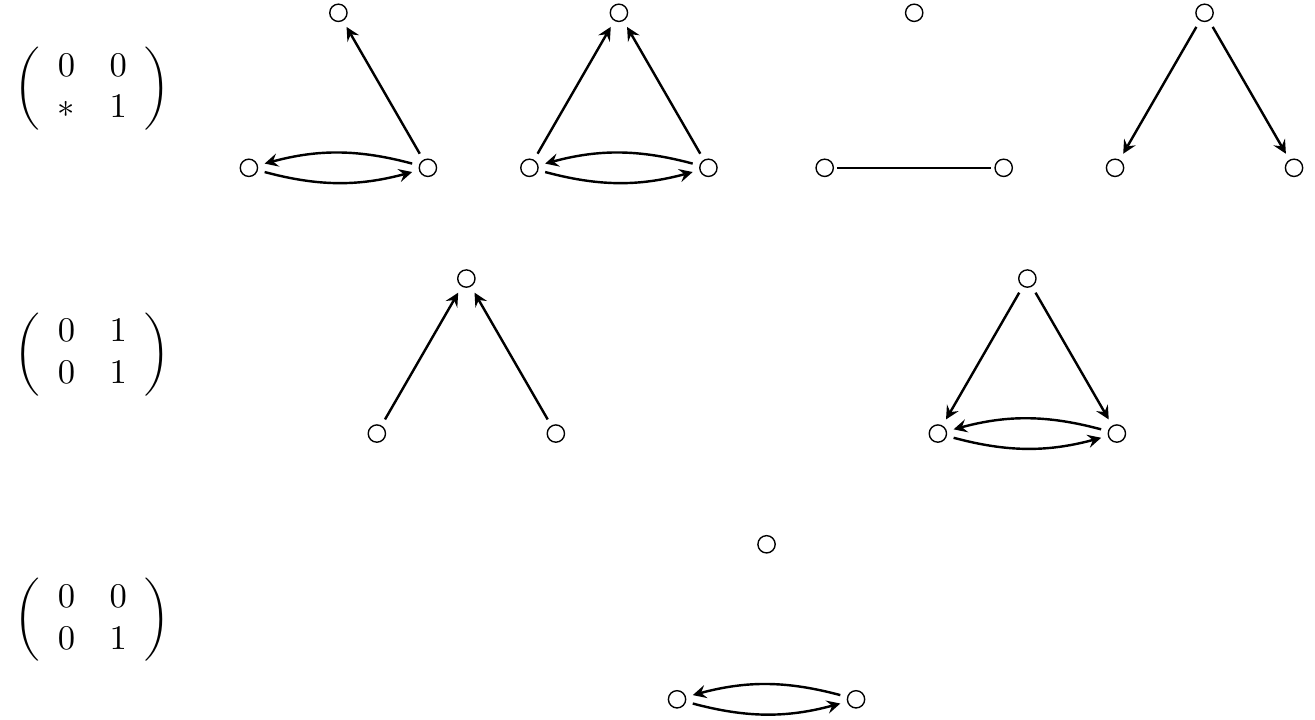}
\end{center}\caption{$M_i$-partitionable digraphs on $3$-vertices for $8 \le i \le 10$.} \label{01DF}
\end{figure}

\begin{theorem} \label{01D}
If $i \in \{8, 9, 10\}$, then the minimal $M_i$-obstructions
are precisely the digraphs in $\mathcal{F}_i$.
\end{theorem}
%
%

\begin{proof}
It is not hard to verify that every minimal
$M_7$-obstruction is an element of
$\mathcal{F}_i$, or properly contains an
element of $\mathcal{F}_i$ as an induced
subdigraph, for $8 \le i \le 10$.   Thus, every
$\mathcal{F}_i$-free digraph is a strict split
digraph, for $8 \le i \le 10$.

Let $D$ be an $\mathcal{F}_8$-free digraph
with strict split partition $(V_0, V_1)$.   If
$V_D$ is a strong clique or an independent
set, then $D$ is $M_8$-partitionable.   Also,
every digraph on $2$ vertices is
$M_8$-partitionable, so let us assume that
$|V_D| \ge 3$ and $V_0 \ne \varnothing \ne
V_1$.   If $|V_1| = 1$, then it follows from the
fact that $D$ is $\mathcal{F}_8$-free that
either all the arcs between $V_0$ and
$V_1$ are oriented towards $V_1$, or
there is an unique arc $(v_1, v_0)$ between
them from $V_1$ to $V_0$.   In this case
$(V_0 \cup \{ v_1 \}, \{ v_0 \})$ is an
$M_8$-partition of $D$.   Otherwise, $|V_1|
\ge 2$, and it is easy to observe that every
arc between $V_0$ and $V_1$ must be
oriented from $V_1$ to $V_0$, because
$D$ is $\mathcal{F}_8$-free.   Thus, $(V_0,
V_1)$ is an $M_8$-partition of $D$.

Let $D$ be an $\mathcal{F}_9$-free digraph
with split partition $(V_0, V_1)$.   Since the
underlying graph of $D$ is $(K_1 + K_2)$-free,
and $D$ is $\mathcal{F}_9$-free, all the arcs
from $V_0$ to $V_1$ must be present in $D$,
and no arc from $V_1$ to $V_0$ can exist.
Therefore, $(V_0, V_1)$ is an $M_9$-partition
of $D$.

Finally, let $D$ be an $\mathcal{F}_{10}$-free
digraph with split partition $(V_0, V_1)$
maximizing the size of $V_1$.   Notice
that the asymmetric arc is an element of
$\mathcal{F}_{10}$, hence, $D$ is a biorientation
of its underlying graph $G_D$.   Thus, it suffices
to notice that $G_D$ is a split graph where the
path of length $2$ is forbidden.   It is easy to
verify that $(V_0, V_1)$ is an $M_{10}$-partition
of $D$, unless there is a vertex in $V_0$ which
is adjacent to every vertex in $V_1$.   But this
cannot happen as it would contradict the choice
of $(V_0, V_1)$.
\end{proof}

\end{document}